\documentclass[english]{amsart}

\usepackage{esint}
\usepackage[svgnames]{xcolor} 
\usepackage[colorlinks,citecolor=red,pagebackref,hypertexnames=false,breaklinks]{hyperref}
\usepackage{pgf,tikz}
\usepackage{pdfsync}

\usepackage{dsfont}
\usepackage{url}
\usepackage[utf8]{inputenc}
\usepackage[T1]{fontenc}
\usepackage{lmodern}
\usepackage{babel}
\usepackage{mathtools}  
\usepackage{amssymb}
\usepackage{lipsum}
\usepackage{mathrsfs}
\usepackage{color}
\usepackage{graphicx}
\usepackage{picture}

\newtheorem{theorem}{Theorem}[section]
\newtheorem{proposition}{Proposition}[section]
\newtheorem{lemma}{Lemma}[section]

\newtheorem{corollary}{Corollary}[section]

\newtheorem{remark}{Remark}[section]

\numberwithin{equation}{section}

\title[Determination of the nonlinear term]{Stable determination of the nonlinear term in a quasilinear elliptic equation by boundary measurements}

\author[Mourad Choulli]{Mourad Choulli}
\address{Universit\'e de Lorraine}
\thanks{The author supported by the grant ANR-17-CE40-0029 of the French National Research Agency ANR (project MultiOnde).}
\email{mourad.choulli@univ-lorraine.fr}

\date{\today}

\subjclass[2010]{35R30} 
\keywords{Determination of the nonlinear term, quasilinear elliptic equation, boundary measurements, localized Dirichlet-to-Neumann map.}

\begin{document}

\begin{abstract}
We establish a Lipschitz stability inequality for the problem of determining the nonlinear term  in a quasilinear elliptic equation by boundary measurements. We give a proof based on a linearization procedure together with special solutions constructed from the fundamental solution of the linearized problem. 
\end{abstract}

\maketitle

\tableofcontents

\section{Introduction}\label{section1}

\subsection{Statement of the problem}

Consider on a bounded domain $\Omega$ of $\mathbb{R}^n$ the quasilinear BVP
\begin{equation}\label{c3}
\left\{
\begin{array}{ll}
\mathrm{div}(a(u)A\nabla u)=0\quad &\mathrm{in}\; \Omega ,
\\
u_{|\Gamma}=f  ,
\end{array}
\right.
\end{equation}
where $a$ is a scalar function and $A$ is a matrix with variable coefficients. Assume that we can define the map $\Lambda_a: f\mapsto \partial_\nu u(f)$ between two well chosen spaces, where $u(f)$ is the solution of the BVP \eqref{c3} when it exists. In the case where $a$ is supposed to be unknown we ask whether $\Lambda_a$ determines uniquely $a$. This problem can be seen as a Calder\'on type problem for the quasilinear BVP \eqref{c3}. We are mainly interested in establishing a stability inequality for this inverse problem.

\subsection{Assumptions and notations}

Throughout this text, $0<\alpha <1$ and $\Omega$ is a  $C^{2,\alpha}$ bounded domain of $\mathbb{R}^n$ $(n\ge 3)$ with boundary $\Gamma$. Fix $A=(a^{ij})\in C^{1,\alpha} (\mathbb{R}^n,\mathbb{R}^{n\times n})$ satisfying $(a^{ij}(x))$ is symmetric for each $x\in \mathbb{R}^n$,
\begin{equation}\label{c1}
\kappa^{-1}|\xi|^2\le A(x)\xi \cdot \xi ,\quad x,\xi \in \mathbb{R}^n, 
\end{equation}
and
\begin{equation}\label{c1.1}
\max_{1\le i,j\le n}\|a^{ij}\|_{L^\infty (\mathbb{R}^n)}\le \kappa,
\end{equation}
where $\kappa >1$ is a given constant.

Pick $\varkappa >0$. Let  $\mu : [0,\infty)\rightarrow [\varkappa,\infty)$ and $\gamma : [0,\infty)\rightarrow [0,\infty)$ be two nondecreasing functions. If $\varrho=\varrho(n,|\Omega|)$ denotes the constant appearing in \eqref{c4} then consider the assumptions

\vskip.1cm
$(\textbf{a1})$  $a\in C^1(\mathbb{R})$, $a\ge \varkappa $ and  $q_a=a'/a \in C^{0,1}(\mathbb{R})$.

\vskip.1cm
$(\textbf{a2})$ $a(z)\le \mu (\varrho^{-1}|z|)$, $z\in \mathbb{R}$.

\vskip.1cm
$(\textbf{a3})$  $|a'(z)|\le \gamma (\varrho^{-1}|z|)$, $z\in \mathbb{R}$.

\vskip.1cm
Let $\tilde{\gamma} : [0,\infty)\rightarrow [0,\infty)$ be another nondecreasing function. We need also to introduce the following assumption, 

\vskip.1cm
$(\bf{\tilde{a}1})$  $a\in C^2(\mathbb{R})$, $a\ge \varkappa $ and  $|a''(z)|\le \tilde{\gamma} (\varrho^{-1}|z|)$, $z\in \mathbb{R}$.

\vskip.1cm
Note that $(\bf{\tilde{a}1})$ implies $(\textbf{a1})$.

\vskip.1cm
The set of functions $a:\mathbb{R}\rightarrow \mathbb{R}$ satisfying $(\textbf{a1})$ (resp. $(\bf{\tilde{a}1}))$, $(\textbf{a2})$ and $(\textbf{a3})$ is denoted hereafter by $\mathscr{A}$ (resp. $\tilde{\mathscr{A}})$.

\vskip.1cm
In the rest of this text $\Gamma_0$ denotes a nonempty open subset of $\Gamma$. We define $H_{\Gamma_0}^{1/2}(\Gamma)$ as follows
\[
H_{\Gamma_0}^{1/2}(\Gamma)=\{f\in H^{1/2}(\Gamma);\; \mathrm{supp}(f)\subset \Gamma_0\}.
\]
We equip $H_{\Gamma_0}^{1/2}(\Gamma)$ with the norm of $H^{1/2}(\Gamma)$.

\vskip.1cm
Also, fix $\Gamma_1\Supset \Gamma_0$ an open subset of $\Gamma$ and $\chi\in C_0^\infty (\Gamma_1)$ so that $\chi=1$ in $\overline{\Gamma}_0$. 

\vskip.1cm
For any $t\in \mathbb{R}$, $\mathfrak{f}_t$  will denote the constant function given by $\mathfrak{f}_t(s)=t$, $s\in \Gamma$. 

\vskip.1cm
Finally, 
\begin{align*}
&C_m=C_m(n,\Omega, \kappa ,\varkappa, \alpha , \mu(m),\gamma(m))>0,\quad m\ge 0,
\\
&C^0_m=C^0_m(n,\Omega, \kappa ,\varkappa, \alpha, m, \mu(m),\gamma(m))>0,\quad m\ge 0,
\\
&C^1_m=C^1_m(n,\Omega, \kappa ,\varkappa,\alpha, m, \mu(m),\gamma(m),\tilde{\gamma}(m),\Gamma_0,\Gamma_1)>0,\quad m\ge 0,
\end{align*} 
will denote generic constants.

\subsection{Main result}

We show in Subsection \ref{sub2.1} that, under assumption $(\textbf{a1})$, for each $f\in C^{2,\alpha}(\Gamma)$ the BVP \eqref{c3} admits a unique solution $u_a=u_a(f)\in C^{2,\alpha}(\overline{\Omega})$. Furthermore, when $a$ satisfies both (\textbf{a1}) and (\textbf{a2}) the Dirichlet-to-Neumann map
\[
\Lambda_a:C^{2,\alpha}(\Gamma)\rightarrow H^{-1/2}(\Gamma):f\mapsto \partial_\nu u_a(f) 
\]
is well defined.

\vskip.1cm
Set $C_{\Gamma_0}^{2,\alpha}(\Gamma)=\{f\in C^{2,\alpha}(\Gamma);\; \mathrm{supp}(f)\subset \Gamma_0\}$ and define the family of localized Dirichlet-to-Neumann maps $(\tilde{\Lambda}_a^t)_{t\in \mathbb{R}}$ as follows
\[
\tilde{\Lambda}_a^t : f\in C_{\Gamma_0}^{2,\alpha}(\Gamma)\mapsto \chi\Lambda_a(\mathfrak{f}_t+f)\in H^{-1/2}(\Gamma),\quad t\in \mathbb{R}.
\]

We will prove in Subsection \ref{sub2.2} that, under the assumption $(\textbf{a3})$, for each $t\in \mathbb{R}$ and $a\in \mathscr{A}$, $\tilde{\Lambda}_a^t$ is Fr\'echet differentiable in a neighborhood of the origin. Furthermore, $d\tilde{\Lambda}_a^t(0)$, the Fr\'echet differential of $\tilde{\Lambda}_a^t$ at $0$, has an extension, still denoted by $d\tilde{\Lambda}_a^t(0)$, belonging to $\mathscr{B}(H_{\Gamma_0}^{1/2}(\Gamma),H^{-1/2}(\Gamma))$, and 
\[
\sup_{|t|\le \tau} \|d\tilde{\Lambda}_a^t(0)\|_{\mathrm{op}}<\infty,\quad \tau >0.
\]
Here and Henceforth $\|\cdot\|_{\mathrm{op}}$ stands for the norm of $\mathscr{B}(H_{\Gamma_0}^{1/2}(\Gamma),H^{-1/2}(\Gamma))$.

\vskip.1cm
We remark that since $C_{\Gamma_0}^{2,\alpha}(\Gamma)$ is dense in $H_{\Gamma_0}^{1/2}(\Gamma)$ the above extension of $d\tilde{\Lambda}_a^t(0)$ is entirely determined by $d\tilde{\Lambda}_a^t(0)$.

\begin{theorem}\label{theoremc1}
For any $a_1,a_2\in \tilde{\mathscr{A}}$ and $\tau>0$ we have
\[
\|a_1-a_2\|_{C([-\tau ,\tau])}\le C_{\tau}^1 \sup_{|t|\le \tau }\|d\tilde{\Lambda}_{a_1}^t(0)-d\tilde{\Lambda}_{a_2}^t(0)\|_{\mathrm{op}}.
\]
\end{theorem}

The following uniqueness result is straightforward from the preceding theorem.

\begin{corollary}\label{corollaryc1}
Let $a_j\in C^2(\mathbb{R})$ satisfies $a_j\ge \varkappa$, $j=1,2$. Then $\tilde{\Lambda}_{a_1}^t=\tilde{\Lambda}_{a_2}^t$ in a neighborhood of the origin for each $t\in \mathbb{R}$ implies $a_1=a_2$.
\end{corollary}

It is worth noticing that if $a_1$ and $a_2$ are as in Corollary \ref{corollaryc1} then $a_1$ and $a_2$ satisfy also $(\textbf{a2})$, $(\textbf{a3})$ and $(\bf{\tilde{a}1})$ with
\[
\mu (\tau)=\max_{|z|\le \tau }\max_{j=1,2}a_j(\varrho z),\quad \gamma (\tau)=\max_{|z|\le \tau }\max_{j=1,2}|a'_j(\varrho z)|,\quad \tau \ge 0,
\]
and
\[
 \tilde{\gamma} (\tau)=\max_{|z|\le \tau }\max_{j=1,2}|a''_j(\varrho z)|,\quad \tau \ge 0.
\]
These $\mu$, $\gamma$ and $\tilde{\gamma}$ depends of course on $a_1$ and $a_2$.

\subsection{Comments}

There are only very few stability inequalities in the literature devoted to the determination of nonlinear terms in quasilinear and semilinear elliptic equations by boundary measurements. The semilinear case was studied in \cite{CHY} by using a method based on linearization together with stability inequality for the problem of determining the potential in a Schr\"odinger equation by boundary measurements. The result in \cite{CHY} was recently improved in \cite{Ch2}. Both quasilinear and semilinear elliptic inverse problems were considered in \cite{Ki} where a method exploiting the singularities of fundamental solutions was used to establish stability inequalities. This method was  used previously in \cite{Ch1} to obtain a stability inequality at the boundary of the conformal factor in an inverse conductivity problem. We show in the present  paper  how we can modify the proof of \cite[(1.2) of Theorem 1.1]{Ch1} to derive the stability inequality stated in Theorem \ref{theoremc1}. The localization argument was inspired by that in \cite{Ki}.

\vskip.1cm
There is a recent rich literature dealing with uniqueness issue concerning the determination of nonlinearities in elliptic equations by boundary measurements using the so-called higher order linearization method. We refer to the recent work \cite{CLLO} and references therein for more details. We also quote without being exhaustive the following references \cite{CFKKU, EPS1, EPS2, IS, IY,KN,KKU, KU1,KU,LLLS1,LLLS, MU, Su, SU} on  semilinear and quasilinear elliptic inverse problems.

\section{Preliminaries}

\subsection{Solvability of the BVP and the Dirichlet-to-Neumann map}\label{sub2.1}

Suppose that $a$ satisfies (\textbf{a1}) and introduce the divergence form quasilinear operator 
\[
Q(x,u,\nabla u)=\mathrm{div}(a(u)A\nabla u),\quad x\in \Omega,\; u\in C^2(\Omega).
\]
The following observation will be crucial in the sequel :  $u\in C^2(\Omega)$ satisfies $Q(x,u,\nabla u)=0$ in $\Omega$ if and only if
\begin{equation}\label{c2}
Q_0(x,u,\nabla u)=\sum_{i,j=1}^n a^{ij}(x)\partial_{ij}^2u(x)+b(x,u,\nabla u)=0,\quad x\in \Omega,
\end{equation}
where
\[ 
b(x,z,p)=B(x)\cdot p+q_a(z) A(x)p\cdot p,\quad  x\in \Omega ,\; z\in \mathbb{R},\;  p\in \mathbb{R}^n,
\]
with
\[
B_j(x)=\sum_{i=1}^n \partial_ia^{ij}(x), \quad x\in \Omega,\; 1\le j\le n.
\]

\vskip.1cm
As $b(\cdot,\cdot ,0)=0$ one can easily check that $Q_0$ satisfies to the assumptions of \cite[Theorem 15.12, page 382]{GT}. Let $f\in C^{2,\alpha}(\Gamma)$. Il light of the observation above we derive that the quasilinear BVP 
\begin{equation}%\label{c3}
\left\{
\begin{array}{ll}
\mathrm{div}(a(u)A\nabla u)=0\quad &\mathrm{in}\; \Omega ,
\\
u_{|\Gamma}=f  ,
\end{array}
\right.
\end{equation}
admits a solution $u_a=u_a(f)\in C^{2,\alpha}(\overline{\Omega})$. The uniqueness of solutions of \eqref{c3} holds from \cite[Theorem 10.7, page 268]{GT} applied to $Q$.

\vskip.1cm
Also, as
\[
a(z)p\cdot A(x)p\ge \varkappa \kappa^{-1}|p|^2
\] 
we infer from \cite[Theorem 10.9, page 272]{GT} that
\begin{equation}\label{c4}
\max_{\overline{\Omega}}|u_a(f)|\le \varrho\max_{\Gamma}|f|,
\end{equation}
where $\varrho=\varrho(n,|\Omega|)$.

\vskip.1cm
On the other hand according to \cite[Theorem 6.8, page 100]{GT}, for each $f\in C^{2,\alpha}(\Gamma)$, there exists a unique $\mathscr{E}f\in C^{2,\alpha}(\overline{\Omega})$ satisfying
\[
\Delta \mathscr{E}f=0\; \mathrm{in}\; \Omega ,\quad \mathscr{E}f_{|\Gamma}=f,
\]
and from \cite[Theorem 6.6, page 98]{GT} we have
\begin{equation}\label{ee}
\|\mathscr{E}f\|_{C^{2,\alpha}(\overline{\Omega})}\le \mathbf{c}\|f\|_{C^{2,\alpha}(\Gamma)},
\end{equation}
where $\mathbf{c}=\mathbf{c}(\Omega,\alpha)>0$.

\vskip.1cm
Assume that $a$ satisfies (\textbf{a1}) and (\textbf{a2}) and let $f\in C^{2,\alpha}(\Gamma)$. Then straightforward computations show that $v=u_a(f)-\mathscr{E}f$ is the solution of the BVP
\begin{equation}\label{c5}
\left\{
\begin{array}{ll}
-\mathrm{div}(a(u_a(f))A\nabla v)=\mathrm{div}(a(u_a(f))A\nabla \mathscr{E}f) \quad &\mathrm{in}\; \Omega ,
\\
v_{|\Gamma}=0  .
\end{array}
\right.
\end{equation}
Multiplying the first equation of \eqref{c5} by $v$ and integrating over $\Omega$. We then obtain from Green's formula
\[
\int_\Omega a(u_a(f))A\nabla v\cdot \nabla vdx=-\int_\Omega a(u_a(f))A\nabla \mathscr{E}f\cdot \nabla vdx.
\]

Set
\[
\mathcal{B}_m=\{ f\in C^{2,\alpha}(\Gamma);\; \max_\Gamma |f|< m\}.
\]

If $f\in \mathcal{B}_m$ then the last identity together with \eqref{c4} yield 
\[
\varkappa \kappa^{-1} \|\nabla v\|_{L^2(\Omega)}^2\le \kappa \mu (m)\|\nabla \mathscr{E}f\|_{L^2(\Omega )}\|\nabla v\|_{L^2(\Omega )},
\]
which, combined with the fact of $w\in H^1(\Omega)\mapsto \|\nabla w\|_{L^2(\Omega)}$ defines an equivalent norm on $H_0^1(\Omega)$, implies
\[
\|v\|_{H^1(\Omega)}\le \aleph\mu (m)\|f\|_{C^{2,\alpha}(\Gamma)}.
\]
Here and henceforth, $\aleph =\aleph(n,\Omega, \kappa,\varkappa,\alpha)>0$ is a generic constant. 

Whence 
\begin{equation}\label{c6}
\|u_a(f)\|_{H^1(\Omega)}\le \aleph\mu (m)\|f\|_{C^{2,\alpha}(\Gamma)}.
\end{equation}

We endow $H^{1/2}(\Gamma)$ with the quotient norm
\[
\|\varphi\|_{H^{1/2}(\Gamma)}=\min\left\{ \|v\|_{H^1(\Omega)};\; v\in \dot{\varphi}\right\},\quad \varphi \in H^{1/2}(\Gamma),
\]
where 
\[
\dot{\varphi}=\left\{v\in H^1(\Omega);\; v_{|\Gamma}=\varphi\right\}.
\]

For each $\psi \in H^{-1/2}(\Gamma)$ we define $\chi\psi$ by 
\[
\langle \chi \psi ,\varphi\rangle_{1/2} = \langle \psi ,\chi \varphi\rangle_{1/2},\quad \varphi\in H^{1/2}(\Gamma),
\]
where $\langle \cdot,\cdot \rangle_{1/2}$ is the duality pairing between $H^{1/2}(\Gamma)$ and its dual $H^{-1/2}(\Gamma)$.

\vskip.1cm
It is not difficult to check that $\chi \psi\in H^{-1/2}(\Gamma)$, $\mathrm{supp}(\chi \psi)\subset \Gamma_1$ and the following identity holds
\begin{equation}\label{ui}
\langle \chi \psi ,\varphi\rangle_{1/2} = \langle \psi , \varphi\rangle_{1/2},\quad \varphi\in H_{\Gamma_0}^{1/2}(\Gamma).
\end{equation}
This identity will be very useful in the sequel.

\vskip.1cm
Let $\varphi \in H^{1/2}(\Gamma )$ and $v\in \dot{\varphi}$. Appying Green's formula, we get
\[
\int_\Gamma a(f)A\nabla u_a(f)\cdot \nu \varphi ds =\int_\Omega a (u_a(f))A\nabla u_a(f)\cdot \nabla v dx,
\]
where $\nu$ is the unit exterior normal vector field on $\Gamma$.

\vskip.1cm
Using that $u_a(f)$ is the solution of the BVP \eqref{c3}, we easily check that the right hand side of the  above identity is independent of $v$, $v\in \dot{\varphi}$.

\vskip.1cm
This identity suggests to define the Dirichlet-to-Neumann map 
\[
\Lambda_a:C^{2,\alpha}(\Gamma)\rightarrow H^{-1/2}(\Gamma), 
\]
associated to $a$, by the formula
\[
\langle \Lambda_a(f),\varphi \rangle_{1/2}=\int_\Omega a (u_a(f))A\nabla u_a(f)\cdot \nabla v dx,\quad \varphi\in H^{1/2}(\Gamma),\; v\in \dot{\varphi}.
\]
Using \eqref{c6}, we get 
\begin{equation}\label{c7}
\|\Lambda_a(f) \|_{H^{-1/2}(\Gamma)}\le \aleph\mu (m)^2\|f\|_{C^{2,\alpha}(\Gamma)},\quad f\in \mathcal{B}_m.
\end{equation}

\subsection{Differentiability properties}\label{sub2.2}

We need a gradient bound for the solution of the BVP \eqref{c3}. To this end we set 
\[
\mathcal{B}_m^+ =\{f\in \mathcal{B}_m  ;\; \|f\|_{C^{2,\alpha}(\Gamma)}< \mathbf{c}^{-1}m\} ,\quad m>0,
\]
where $\mathbf{c}$ is the constant in \eqref{ee}.

\vskip.1cm
Fix $f\in \mathcal{B}_m^+$ and $a\in \mathscr{A}$. We apply \cite[Theorem 15.9, page 380]{GT} with $\varphi=\mathscr{E}f/\|\mathscr{E}f\|_{C^2(\overline{\Omega})}$ and  $u=u_a(f)/\|\mathscr{E}f\|_{C^2(\overline{\Omega})}$ which is the solution of \eqref{c3}  when $a(z)$ is substituted by $a(\|\mathscr{E}f\|_{C^2(\overline{\Omega})}z)$. We obtain
\begin{equation}\label{ge}
\max_{\overline{\Omega}}|\nabla u_a(f)|\le C^0_m\|f\|_{C^{2,\alpha}(\Gamma)}.
\end{equation}

\vskip.1cm
Next, we establish that $\Lambda_a$, $a\in \mathscr{A}$, is Fr\'echet differentiable in a neighborhood of the origin. For $\eta >0$ define 
\[
\mathcal{B}_m^\eta =\{f\in \mathcal{B}_m^+\cap C_{\Gamma_0}^{2,\alpha}(\Gamma) ;\; \|f\|_{C^{2,\alpha}(\Gamma)}<\eta\} .
\]

\begin{lemma}\label{lemmac1}
Let $m>0$. There exists $\eta_m =\eta_m (n,\Omega, \kappa ,\varkappa,\alpha, m, \mu(m),\gamma(m))>0$ such that for each $a\in \mathscr{A}$ we have
\[
\|u_a(f)-u_a(g)\|_{H^1(\Omega)}\le C_m\|f-g\|_{C^{2,\alpha}(\Gamma)},\quad f,g\in \mathcal{B}_m^{\eta_m},
\]
\end{lemma}

\begin{proof}
Let $\eta >0$ to be determined later. Pick $f,g\in \mathcal{B}_m^\eta$ and set $h=g-f$. Let $\sigma =a(u_a(g))$ and
\[
p(x)=\int_0^1a'(u_a(f)(x)+t[u_a(g)(x)-u_a(f)(x)])dt,\quad x\in \Omega.
\]
It is then straightforward to check that $u=u_a(g)-u_a(f)$ is the solution of the BVP
\[
\left\{
\begin{array}{ll}
-\mathrm{div}(\sigma A\nabla u)=\mathrm{div}(puA\nabla u_a(f)) \quad &\mathrm{in}\; \Omega ,
\\
u_{|\Gamma}=h  .
\end{array}
\right.
\]
 
We split $u$ into two terms $u=\mathscr{E}h+v$, where $v$ is the solution of the BVP
\[
\left\{
\begin{array}{ll}
-\mathrm{div}(\sigma A\nabla v)-\mathrm{div}(pvA\nabla u_a(f))= \mathrm{div}(F)\quad &\mathrm{in}\; \Omega ,
\\
v_{|\Gamma}=0  ,
\end{array}
\right.
\]
with 
\[
F=\sigma A\nabla \mathscr{E}h+q\mathscr{E}hA\nabla u_a(f).
\]
Applying Green's formula, we find
\[
\int_\Omega \sigma A\nabla v\cdot \nabla v+\int_\Omega qvA\nabla u_a(f)\cdot \nabla v=-\int_\Omega F\cdot \nabla v.
\]

From  \eqref{ge} and Poincar\'e's inequality we derive
\[
\varkappa \kappa^{-1}\|\nabla v\|^2_{L^2(\Omega)}-C_m^0\eta \|\nabla v\|_{L^2(\Omega)}^2\le C_m\|\nabla v\|_{L^2(\Omega)}\|h\|_{C^{2,\alpha}(\Gamma)}.
\]

If $\eta=\eta_m$ is chosen sufficiently small is such a way that 
\[
\varkappa \kappa^{-1}-C_m^0\eta \ge \varkappa \kappa^{-1}/2,
\]
 then we obtain
\[
\|\nabla v\|_{L^2(\Omega)}\le C_m\|h\|_{C^{2,\alpha}(\Gamma)},
\]
from which the expected inequality follows readily.
\end{proof}

In the sequel $\eta_m$, $m>0$, will denote the constant in Lemma \ref{lemmac1}.

\begin{lemma}\label{lemmac2}
Pick $a\in \mathscr{A}$ and $m>0$. Then
\[
\|\Lambda_a (f)-\Lambda_a(g)\|_{H^{-1/2}(\Gamma)}\le C_m\|f-g\|_{C^{2,\alpha}(\Gamma)},\quad f,g\in \mathcal{B}_m^{\eta_m}.
\]
\end{lemma}

\begin{proof}
Let $f,g\in \mathcal{B}_m^{\eta_m}$. For $\varphi\in H^{1/2}(\Gamma)$ and $v\in \dot{\varphi}$, we have
\[
\langle \Lambda_a(g)-\Lambda_a(f),\varphi \rangle_{1/2}= I_1+I_2,
\]
where
\begin{align*}
&I_1= \int_\Omega [a (u_a(g))-a(u_a(f))]A\nabla u_a(g)\cdot \nabla v dx,
\\
&I_2=\int_\Omega a (u_a(f))A[\nabla u_a(g)-\nabla u_a(f)]\cdot \nabla vdx.
\end{align*}
We can proceed similarly to the proof of Lemma \ref{lemmac1} to derive that
\[
|I_j|\le C_m\|u_a(g)-u_a(f)\|_{H^1(\Omega)}\|v\|_{H^1(\Omega)},\quad j=1,2.
\]
The expected inequality follows easily by using Lemma \ref{lemmac1}.
\end{proof}

Let $f\in \mathcal{B}_m^{\eta_m}$. Similarly to the calculations we carried out in the proof of Lemma \ref{lemmac1}, we show that the bilinear continuous form
\[
\mathfrak{b}(v,w)=\int_\Omega [a(u_a(f))A\nabla v+a'(u_a(f))vA\nabla u_a(f)] \cdot \nabla wdx,\quad v,w\in H_0^1(\Omega),
\]
is coercive. In light of Lemma \ref{lemmaap2}, we obtain that the BVP
\begin{equation}\label{c8}
\left\{
\begin{array}{ll}
\mathrm{div}[a(u_a(f))A\nabla v+a'(u_a(f))v\nabla u_a(f))]= 0\quad \mathrm{in}\; \Omega ,
\\
v_{|\Gamma}=h  ,
\end{array}
\right.
\end{equation}
admits a unique weak solution $v_a=v_a(f,h)\in H^1(\Omega)$ satisfying
\begin{equation}\label{c9.0}
\|v_a(f,h)\|_{H^1(\Omega)}\le C_m\|h\|_{H^{1/2}(\Gamma)}
\end{equation}
and hence
\begin{equation}\label{c9}
\|v_a(f,h)\|_{H^1(\Omega)}\le C_m\|h\|_{C^{2,\alpha}(\Gamma)}.
\end{equation}
We refer to Appendix \ref{appendix} for the exact definition of weak solutions.

\vskip.1cm
Next, pick $\epsilon >0$ such that $f+h\in \mathcal{B}_m^{\eta_m}$ for each $h\in C_{\Gamma_0}^{2,\alpha}(\Gamma)$ satisfying $\|h\|_{C^{2,\alpha}(\Gamma)}<\epsilon$. Set then
\[
w=u_a(f+h)-u_a(f)-v_a(f,h).
\]
Simple computations show that $w$ is the weak solution of the BVP
\[
\left\{
\begin{array}{ll}
\mathrm{div}[a(u_a(f))A\nabla w+a'(u_a(f))w\nabla u_a(f))]= \mathrm{div}(F)\quad &\mathrm{in}\; \Omega ,
\\
w_{|\Gamma}=0 ,
\end{array}
\right.
\]
with 
\begin{align*}
F&= a'(u_a(f))[u_a(f+h))-u_a(f)]\nabla u_a(f)
\\
&\hskip 4cm -[a(u_a(f+h))-a(u_a(f))]\nabla u_a(f+h)
\\
&=\{a'(u_a(f))[u_a(f+h))-u_a(f)]-[a(u_a(f+h))-a(u_a(f))]\}\nabla u_a(f)
\\
&\hskip 3.5cm+[a(u_a(f+h))-a(u_a(f))][\nabla u_a(f)-\nabla u_a(f+h)].
\end{align*}
In particular, we have
\begin{equation}\label{c10}
\mathfrak{b}(w,w)=\int_\Omega F\cdot \nabla w.
\end{equation}
Using that
\begin{align*}
a'(&u_a(f))[u_a(f+h))-u_a(f)]-[a(u_a(f+h))-a(u_a(f))]
\\
&=[u_a(f+h))-u_a(f)]\int_0^1[a'(u_a(f))-a'(u_a(f)+t(u_a(f+h)-u_a(f))]dt,
\end{align*}
and  the uniform continuity of $a'$ in $[-\varrho m,\varrho m]$, we obtain
\[
\|a'(u_a(f))[u_a(f+h))-u_a(f)]-[a(u_a(f+h))-a(u_a(f))]\|_{L^\infty(\Omega)}=o(\|h\|_{C^{2,\alpha}(\Gamma)}).
\]
On the other hand similar estimates as above give
\[
\|[a(u_a(f+h))-a(u_a(f))][\nabla u_a(f)-\nabla u_a(f+h)]\|_{H^1(\Omega)}\le C_m\|h\|_{C^{2,\alpha}(\Gamma)}^2.
\]

The last two inequalities together with \eqref{c10} yield
\[
\|w\|_{H^1(\Omega)}= o(\|h\|_{C^{2,\alpha}(\Gamma)}).
\]

In other words we proved that $f\in \mathcal{B}_m^{\eta_m}\mapsto u_a(f)\in H^1(\Omega)$ is Fr\'echet differentiable with 
\[
du_a(f)(h)=v_a(f,h),\quad f\in \mathcal{B}_m^{\eta_m},\; h\in C_{\Gamma_0}^{2,\alpha}(\Gamma).
\]
Using the definition of $\Lambda_a$ we can then state the following result.

\begin{proposition}\label{propositionc1}
For each $m>0$, the mapping \[f\in \mathcal{B}_m^{\eta_m}\mapsto \Lambda_a(f)\in H^{-1/2}(\Gamma)\] is Fr\'echet differentiable with 
\[
\langle d\Lambda_a(f)(h),\varphi \rangle_{1/2}=\int_\Omega [a(u_a(f))A\nabla v_a(f,h)+a'(u_a(f))v_a(f,h)A\nabla u_a(f)]\cdot \nabla v dx,
\]
$f\in \mathcal{B}_m^{\eta_m}$, $h\in C_{\Gamma_0}^{2,\alpha}(\Gamma)$, $\varphi\in H^{1/2}(\Gamma)$ and $v\in \dot{\varphi}$.
\end{proposition}

The fact that  $d\Lambda_a(f)$, $f\in \mathcal{B}_m^{\eta_m}$,  is extended to a bounded linear map from $H_{\Gamma_0}^{1/2}(\Gamma)$ into $H^{-1/2}(\Gamma)$ is  immediate from \eqref{c9.0}. 

\begin{remark}\label{remark1}
{\rm
If the assumption (\textbf{a1}) is substituted by the following one

\vskip.1cm
(\textbf{a1'})  $a\in C^{1,1}(\mathbb{R})$, $a\ge \varkappa $,

\vskip.1cm
\noindent
then one can prove that $f\in \mathcal{B}_m^{\eta_m}\mapsto v_a(f,\cdot)\in \mathscr{B}(C_{\Gamma_0}^{2,\alpha}(\Gamma),H^1(\Omega))$ is continuous. In that case $f\in \mathcal{B}_m^{\eta_m}\mapsto \Lambda_a(f)\in H^{-1/2}(\Gamma)$ is continuously  Fr\'echet differentiable.
}
\end{remark}

\section{Proof of the main result}\label{section3}

As we already mentioned we give a  proof based on an adaptation of  \cite[proof of (1.2) of Theorem 1.1]{Ch1} combined with a localization argument borrowed from \cite{Ki}.

\subsection{Special solutions}

We construct in a general setting special solutions of a divergence form operator vanishing outside $\Gamma_0$. To this end, let $\mathfrak{A}=(\mathfrak{a}^{ij})\in C^{1,\alpha}(\mathbb{R}^n,\mathbb{R}^{n\times n})$ satisfying 
\[
\lambda^{-1} |\xi|^2\le \mathfrak{A}(x)\xi \cdot \xi ,\quad x\in \mathbb{R}^n ,\; \xi \in \mathbb{R}^n,
\]
and
\[
\max_{1\le i,j\le n}\|\mathfrak{a}^{ij}\|_{C^{1,\alpha}(\mathbb{R}^n)}\le \lambda ,
\]
for some constant $\lambda >1$.

\vskip.1cm
Recall that  the canonical parametrix for the operator $\mathrm{div}(\mathfrak{A}\nabla \cdot\, )$ is given by
\[
H(x,y)=\frac{[\mathfrak{A}^{-1}(y)(x-y)\cdot (x-y)]^{(2-n)/2}}{(n-2)|\mathbb{S}^{n-1}[\mathrm{det}(\mathfrak{A}(y)]^{1/2}},\quad x,y\in \mathbb{R}^n,\; x\ne y.
\]

\begin{theorem}\label{theorema1}
$($\cite[Theorem 3, page 271]{Ka}$)$ 
Pick $\Omega_0\supset \Omega$. For each $y\in \overline{\Omega}_0$, there exists $u_y\in C^2(\overline{\Omega}_0\setminus \{y\})$ satisfying $\mathrm{div}(\mathfrak{A}\nabla u)=0$ in $\Omega_0\setminus \{y\}$,
\begin{align*}
&|u_y(x)-H(x,y)|\le C|x-y|^{2-n+\alpha},\quad x\in \overline{\Omega}_0\setminus \{y\},
\\
&|\nabla u_y(x)-\nabla H(x,y)|\le C|x-y|^{1-n+\alpha},\quad x\in \overline{\Omega}_0\setminus \{y\},
\end{align*}
where $C=C(n,\Omega_0 , \alpha ,\lambda)>0$.
\end{theorem}

Pick $x_0\in \Gamma_0$ and let $r_0>0$ sufficiently small in such a way that $B(x_0,r_0)\cap \Gamma\Subset \Gamma_0$. As $B(x_0,r_0) \setminus \overline{\Omega}$ contains a cone with a vertex at $x_0$, we find $\delta_0>0$ and a vector $\xi \in \mathbb{S}^{n-1}$ such that, for each $0<\delta \le \delta_0$, we have $y_\delta=x_0+\delta \xi \in B(x_0,r_0)\setminus \overline{\Omega}$ and $\mathrm{dist}(y_\delta,\overline{\Omega})\ge c\delta$, for some constant $c=c(\Omega)>0$. 

\vskip.1cm
In the sequel $\Omega_0=\Omega \cup B(x_0,r_0)$ and $u_\delta =u_{y_\delta}$, $0<\delta \le \delta_0$, where $u_{y_\delta}$ is given by Theorem \ref{theorema1}. Reducing $\delta_0$ if necessary, we may assume that
\begin{equation}\label{bd}
\mathrm{dist}(y_\delta,\partial \Omega_0)\ge r_0/2,\quad 0<\delta \le \delta_0.
\end{equation}

Let $P\in L^\infty(\mathbb{R}^n,\mathbb{R}^n)$ satisfying $\|P\|_{L^\infty(\mathbb{R}^n)}\le \lambda$ and consider on $H_0^1(\Omega)\times H_0^1(\Omega)$ the continuous bilinear forms
\begin{align*}
&\mathfrak{b}(u,v)=\int_\Omega [\mathfrak{A}\nabla u+uP]\cdot \nabla v,\quad u,v\in H_0^1(\Omega).
\\
&\mathfrak{b}^\ast (u,v)=\int_\Omega [\mathfrak{A}\nabla u\cdot \nabla v-vP\cdot \nabla u],\quad u,v\in H_0^1(\Omega).
\end{align*}
We assume that $\mathfrak{b}$ and $\mathfrak{b}^\ast$ are coercive: there exists $c_0>0$ such that 
\[
\mathfrak{b}(u,u)\ge c_0\|u\|_{H_0^1(\Omega)},\quad\mathfrak{b}^\ast(u,u)\ge c_0\|u\|_{H_0^1(\Omega)} \quad u\in H_0^1(\Omega).
\]

Let  $f\in H^{1/2}(\Gamma)$. From Lemma \ref{lemmaap2} and its proof  the BVP
\begin{equation}\label{a1}
\left\{
\begin{array}{ll}
\mathrm{div}(\mathfrak{A}\nabla u+uP)= 0\quad \mathrm{in}\; \Omega ,
\\
u_{|\Gamma}=f  ,
\end{array}
\right.
\end{equation}
has a unique weak solution $u(f)\in H^1(\Omega)$ satisfying
\begin{equation}\label{a2}
\|u(f)\|_{H^1(\Omega)}\le C\|f\|_{H^{1/2}(\Gamma)},
\end{equation}
where $C=C(n,\Omega ,\lambda, c_0)>0$.

\vskip.1cm
Similarly, the BVP
\begin{equation}\label{a1.1}
\left\{
\begin{array}{ll}
\mathrm{div}(\mathfrak{A}\nabla u^\ast)-P\cdot \nabla u^\ast= 0\quad \mathrm{in}\; \Omega ,
\\
u^\ast_{|\Gamma}=f  ,
\end{array}
\right.
\end{equation}
admits unique weak solution $u^\ast(f)\in H^1(\Omega)$ satisfying
\begin{equation}\label{a2.1}
\|u^\ast(f)\|_{H^1(\Omega)}\le C\|f\|_{H^{1/2}(\Gamma)},
\end{equation}
where $C$ is as in \eqref{a2}.

\vskip.1cm
On the other hand, using that the continuous bilinear form
\[
\mathfrak{b}_0(u,v)=\int_{\Omega_0}\mathfrak{A}\nabla u\cdot \nabla v,\quad u,v\in H_0^1(\Omega_0),
\]
is coercive, we derive that the BVP
\begin{equation}\label{a3}
\left\{
\begin{array}{ll}
\mathrm{div}(\mathfrak{A}\nabla v)= 0\quad \mathrm{in}\; \Omega_0 ,
\\
v_{|\partial \Omega_0}=u_\delta  ,
\end{array}
\right.
\end{equation}
admits a unique weak solution $v_\delta\in H^1(\Omega_0)$ satisfying
\begin{equation}\label{e1}
\|v_\delta\|_{H^1(\Omega_0)}\le C\|u_\delta\|_{H^{1/2}(\partial \Omega_0)},\quad 0<\delta\le \delta_0.
\end{equation}
where $C=C(n,\Omega_0 ,\lambda)>0$.

\begin{lemma}\label{lemmasp1}
We have
\begin{equation}\label{a4}
\|v_\delta\|_{H^{1/2}(\Gamma)}\le C, \quad 0<\delta\le \delta_0,
\end{equation}
where $C=C(n,\Omega,\alpha ,\lambda, x_0,r_0)>0$.
\end{lemma}

\begin{proof}
Let $\tilde{\Omega}=\{x\in \Omega_0;\; \mathrm{dist}(x,\partial \Omega_0)<r_0/4\}$. By the continuity of the trace operator, we have
\begin{equation}\label{e2}
\|u_\delta\|_{H^{1/2}(\partial \Omega_0)}\le C\|u_\delta\|_{H^1(\tilde{\Omega})},\quad 0<\delta\le \delta_0,
\end{equation}
for some constant $C=C(\Omega_0,r_0)>0$. Then in light of the inequality
\[
\|u_\delta\|_{H^1(\tilde{\Omega})}\le \|u_\delta-H(\cdot,y_\delta )\|_{H^1(\tilde{\Omega})}+\|H(\cdot ,y_\delta)\|_{H^1(\tilde{\Omega})}, \quad 0<\delta\le \delta_0,
\]
\eqref{bd} and Theorem \ref{theorema1} we obtain
\begin{equation}\label{e3}
\|u_\delta\|_{H^1(\tilde{\Omega})}\le C, \quad 0<\delta\le \delta_0,
\end{equation}
where $C=C(n,\Omega,\alpha ,\lambda, x_0,r_0)>0$. 

\vskip.1cm
A combination of \eqref{e1}, \eqref{e2} and \eqref{e3} then implies
\[
\|v_\delta\|_{H^{1/2}(\Gamma)}\le C, \quad 0<\delta\le \delta_0,
\]
where $C$ is as above.
\end{proof}

Set $f_\delta =(u_\delta-v_\delta)_{|\Gamma} \in H^{1/2}(\Gamma)$, $0<\delta\le \delta_0$. By construction we have $v_\delta=u_\delta$ on $\partial \Omega_0\cap \Gamma\Supset \Gamma \setminus \overline{\Gamma}_0$. Hence $\mathrm{supp}(f_\delta)\subset \Gamma_0$ (see figure 1 below). That is we have $f_\delta \in H_{\Gamma_0}^{1/2}(\Gamma)$.

\vskip 1cm
\begin{center}

\setlength{\unitlength}{0.75cm}
\begin{picture}(6,4)
\linethickness{0.075mm}
\thinlines
\put(2,3){\oval(4,2)}
\put(2,3){$\Omega$}
\put(-0.5,3){$\Gamma$}
\put(2.5,4.2){{\color{blue}$\Gamma_0$}}
\put(4,3){{\color{red}\circle{1.1}}}
\put(5,3){{\color{red}$B(x_0,r_0)$}}

\thicklines
{\color{blue}
\put(2,3){\oval(4,2)[tr]}
\put(2,3){\oval(4,2)[pr]}
}
\put(1.3,1.2){Figure 1}
\end{picture}

\end{center}

Until the end of this subsection $C=C(n,\Omega, \lambda ,\alpha, c_0 ,x_0,r_0)>0$ denotes a generic constant.

\vskip.1cm
For $\delta >0$ define
\[
\ell_n(\delta)=\left\{ \begin{array}{ll}\displaystyle  1, &n=3,\\ \displaystyle |\ln \delta|^{1/2}, &n=4,\\ \displaystyle \delta^{2-n/2}, \quad &n\ge 5. \end{array}\right. \
\]

\begin{lemma}\label{lemmasp2}
Let $0<\delta\le \delta_0$ and denote by $w_\delta\in H^1(\Omega)$ the weak solution of the BVP \eqref{a1} when $f=f_\delta$. Then $w_\delta=H(\cdot,y_\delta)+z_\delta$ with
\[
\|z_\delta \|_{H^1(\Omega)}\le C\ell_n(\delta).
\]
\end{lemma}

\begin{proof}
We first note that $\tilde{z}_\delta=w_\delta -u_\delta\in H^1(\Omega)$  is the weak solution of the BVP
\[
\left\{
\begin{array}{ll}
-\mathrm{div}(\mathfrak{A}\nabla \tilde{z}+\tilde{z}P)= \mathrm{div}(u_\delta P)\quad \mathrm{in}\; \Omega ,
\\
\tilde{z}_{|\Gamma}=-v_\delta  ,
\end{array}
\right.
\]
It follows from Lemma \ref{lemmaap3} that
\[
\|\tilde{z}_\delta\|_{H^1(\Omega)}\le C\left(\|v_\delta\|_{H^{1/2}(\Gamma)}+\|u_\delta\|_{L^2(\Omega)} \right).
\]
Using that $\mathrm{dist}(y_\delta,\overline{\Omega})\ge c\delta$ (and hence $\Omega\subset B(R,y_\delta)\setminus B(y_\delta,c\delta/2)$ for some large $R>0$ independent on $\delta$) we easily derive from Theorem \ref{theorema1}
\[
\|u_\delta\|_{L^2(\Omega)}\le C\ell_n(\delta).
\]
We combine this estimate with \eqref{a4} in order to obtain
\begin{equation}\label{a4.1}
\|\tilde{z}_\delta \|_{H^1(\Omega)}\le C\ell_n(\delta).
\end{equation}

Let $z_\delta=\tilde{z}_\delta +u_\delta -H(\cdot,y_\delta)$. Then we have the decomposition $w_\delta=H(\cdot,y_\delta)+z_\delta$. Using once again Theorem \ref{theorema1} and \eqref{a4.1}, we obtain
\[
\|z_\delta \|_{H^1(\Omega)}\le C\ell_n(\delta)
\]
as expected.
\end{proof}

\begin{lemma}\label{lemmasp3}
Assume that $P\in W^{1,\infty}(\Omega)$ with $\|P\|_{W^{1,\infty}(\Omega)}\le \lambda$. Let $0<\delta\le \delta_0$ and denote by $w_\delta^\ast\in H^1(\Omega)$ the weak solution of the BVP \eqref{a1.1} when $f=f_\delta$. Then $w_\delta=H(\cdot,y_\delta)+z^\ast_\delta$ with
\[
\|z^\ast_\delta \|_{H^1(\Omega)}\le C\ell_n(\delta).
\]
\end{lemma}

\begin{proof}
As $\tilde{z}^\ast_\delta =w_\delta^\ast-u_\delta$ is the solution of the BVP 
\[
\left\{
\begin{array}{ll}
\mathrm{div}(\mathfrak{A}\nabla \tilde{z}^\ast)-P\cdot \nabla \tilde{z}^\ast= P\cdot \nabla u_\delta \quad \mathrm{in}\; \Omega ,
\\
\tilde{z}^\ast_{|\Gamma}=-v_\delta  ,
\end{array}
\right.
\]
we obtain from Lemma \eqref{lemmaap4}
\[
\|\tilde{z}^\ast_\delta\|_{H^1(\Omega)}\le C\left(\|v_\delta\|_{H^{1/2}(\Gamma)}+\|u_\delta\|_{L^2(\Omega)}\right).
\]
The rest of the proof is similar to that of Lemma \ref{lemmasp2}.
\end{proof}

\subsection{Stability of determining the conformal factor at the boundary}

Suppose that $\mathfrak{A}=\sigma A$, where $A$ is as in Section \ref{section1} and $\sigma \in C^{1,\alpha}(\mathbb{R}^n)$,  define 
\[\Lambda_\sigma: H^{1/2}(\Gamma)\rightarrow H^{-1/2}(\Gamma)\] 
as follows
\[
\langle \Lambda_\sigma (f),\varphi \rangle_{1/2} =\int_\Omega[\sigma A\nabla u_\sigma (f)+u_\sigma (f)P]\cdot \nabla v,\quad \varphi\in H_{\Gamma_0}^{1/2}(\Gamma),\; v\in \dot{\varphi}.
\]
where $u_\sigma (f)$ is the solution of \eqref{a1} when $\mathfrak{A}=\sigma A$. We also consider
\[
\tilde{\Lambda}_\sigma : f\in H_{\Gamma_0}^{1/2}(\Gamma)\mapsto \chi \Lambda_\sigma (f)\in H^{1/2}(\Gamma).
\]

Pick $P_j$, $j=1,2$ satisfying the  assumptions  of Lemma \ref{lemmasp3}.

\vskip.1cm
Let $u_j=u_{\sigma_j}$ when $P=P_j$, $\Lambda_j=\Lambda_{\sigma_j}$, $j=1,2$. Set then  $\sigma=\sigma_1-\sigma_2$, $P=P_1-P_2$ and $u=u_1-u_2$. With these notations we have
\begin{align*}
\langle (\Lambda_1-\Lambda_2) (f),v_{|\Gamma}\rangle_{1/2} =\int_\Omega \sigma &A\nabla u_1 (f)\cdot \nabla v+\int_\Omega \sigma_2A\nabla u\cdot \nabla v
\\
&+\int_\Omega [u_1(f)P-uP_2]\cdot \nabla v,\quad v\in H^1(\Omega).
\end{align*}

We use this identity with $v=v^\ast(g)$, $g\in H^{1/2}(\Gamma)$, the weak solution of the BVP
\[
\left\{
\begin{array}{ll}
\mathrm{div}(\sigma_2A\nabla v^\ast)-P_2\cdot \nabla v^\ast= 0\quad \mathrm{in}\; \Omega ,
\\
v^\ast_{|\Gamma}=g  .
\end{array}
\right.
\]
Since
\begin{align*}
\int_\Omega \sigma_2A\nabla u\cdot \nabla v^\ast(g)dx-&\int_\Omega uP_2\cdot \nabla v^\ast(g)dx
\\
&=-\int_\Omega u\mathrm{div}(\sigma_2\nabla v^\ast(g))dx-\int_\Omega uP_2\cdot \nabla v^\ast(g)dx=0
\end{align*}
we obtain by taking into account \eqref{ui}
\begin{align*}
\langle (\tilde{\Lambda}_1-\tilde{\Lambda}_2) (f),g\rangle_{1/2} =\int_\Omega \sigma & A\nabla u_1 (f)\cdot \nabla v^\ast(g)dx
\\
&+\int_\Omega u_1(f)P\cdot\nabla v^\ast(g)dx,\quad f,g\in H_{\Gamma_0}^{1/2}(\Gamma).
\end{align*}

Let $H_j=H$ when $\mathfrak{A}=\sigma_jA$, $j=1,2$. That is we have
\[
H_j(x,y)=\frac{[A^{-1}(y)(x-y)\cdot (x-y)]^{(2-n)/2}}{(n-2)|\mathbb{S}^{n-1}|\sigma(y)[\mathrm{det}(A(y)]^{1/2}},\quad x,y\in \mathbb{R}^n,\; x\ne y.
\]

According to Lemmas \ref{lemmasp2} and \ref{lemmasp3}, with $f=g=f_\delta$,  $\mathfrak{A}=\sigma_jA$ and $P=P_j$, $j=1$ or $j=2$, we have
\begin{align*}
&u_1(f_\delta) =H_1(\cdot ,y_\delta)+z_\delta,\quad 0<\delta\le \delta_0,
\\
&v^\ast(f_\delta) =H_2(\cdot ,y_\delta)+z_\delta^\ast,\quad 0<\delta\le \delta_0,
\end{align*}
where $z_\delta $ and $z_\delta^\ast$ satisfies
\[
\|z_\delta\|_{H^1(\Omega)}\le C\ell_n(\delta),\quad \|z_\delta^\ast\|_{H^1(\Omega)}\le C\ell_n(\delta),\quad 0<\delta\le \delta_0.
\]

In the rest subsection we always need to reduce $\delta_0$. For simplicity convenience we keep the notation $\delta_0$.

\vskip.1cm
Fix $\Upsilon$ a nonempty closed subset of $\Gamma_0$ and assume that $\|\sigma\|_{C(\Upsilon)}=\sigma (x_0)$. Proceeding as in the proof \cite[(2.8)]{Ch1}, we get
\[
C\|\sigma\|_{C(\Upsilon)}\le \delta^{n-2}\int_\Omega \sigma  A\nabla u_1 (f_\delta)\cdot \nabla v^\ast(f_\delta)dx+\delta^\alpha
\]
We also prove in a similar manner
\begin{align*}
&\left|\int_\Omega u_1(f)P\cdot\nabla v^\ast(f_\delta)dx\right|\le C\ell_n(\delta)\delta^{1-n/2},
\\
&\left|\langle (\tilde{\Lambda}_1-\tilde{\Lambda}_2) (f_\delta),v^\ast(f_\delta)\rangle_{1/2}\right|\le C'\|\tilde{\Lambda}_1-\tilde{\Lambda}_2\|_{\mathrm{op}} \delta^{2-n}.
\end{align*}
Here and in the sequel $C'=C'(n,\Omega, \lambda ,\alpha ,\Gamma_0,\Gamma_1,x_0,r_0)>0$ denotes a generic constant.

\vskip.1cm
Hence
\[
C'\|\sigma\|_{C(\Upsilon)}\le \|\tilde{\Lambda}_1-\tilde{\Lambda}_2\|_{\mathrm{op}}+\max(\delta^{n/2-1}\ell_n(\delta),\delta^\alpha),\quad 0<\delta\le \delta_0,
\]
from which we derive
\begin{equation}\label{st0}
\|\sigma\|_{C(\Upsilon)}\le C'\|\tilde{\Lambda}_1-\tilde{\Lambda}_2\|_{\mathrm{op}}.
\end{equation}

\subsection{Proof of Theorem \ref{theoremc1}}

Let $a_1,a_2\in \tilde{\mathscr{A}}$. We apply the preceding result with 
\[
\sigma_j=a_j(u_{a_j}(f)),\quad P_j=a'(u_{a_j}(f))\nabla u_{a_j}(f),\quad f\in B_m^{\eta_m},\; j=1,2, \; t\in \mathbb{R}. 
\]
Note that the coercivity of $\mathfrak{b}$ and $\mathfrak{b}^\ast$ when $P=P_j$, $j=1,2$, was already demonstrated in the previous section, with coercivity constant independent on $f\in B_m^{\eta_m}$.

\vskip.1cm
By taking $f=0$ we easily get from \eqref{st0}
\begin{equation}\label{st}
|a_1(0)-a_2(0)|\le C_0^1 \|d\tilde{\Lambda}_{a_1}^0(0)-d\tilde{\Lambda}_{a_2}^0(0)\|_{\mathrm{op}}.
\end{equation}

For each $a\in \mathscr{A}$ and $t \in \mathbb{R}$, we obtain by straightforward computations that
\[
u_{a^t}(f)=u_a(f+\mathfrak{f}_t)-\mathfrak{f}_t,
\]
where $a^t (z)=a(z+t)$, $z\in \mathbb{R}$. This identity yields
\begin{equation}\label{ti}
\Lambda_a(f+\mathfrak{f}_t)=\Lambda_{a^t}(f),\quad f\in C_{\Gamma_0}^{2,\alpha}(\Gamma).
\end{equation}

The following assumptions hold for the family  $(a^t)_{t\in \mathbb{R}}$ : for any $\tau >0$, we have
\begin{align*}
&a^t(z)\le \mu_\tau (\varrho ^{-1}|z|)=\mu (\varrho ^{-1}(|z|+\tau )) ,\quad a\in \mathbb{R},\; |t|\le \tau ,
\\
&|(a^t)'(z)|\le \gamma_\tau (\varrho ^{-1}|z|)=\gamma (\varrho ^{-1}(|z|+\tau )) ,\quad a\in \mathbb{R},\; |t|\le \tau ,
\\
&|(a^t)''(z)|\le \tilde{\gamma}_\tau (\varrho ^{-1}|z|)=\tilde{\gamma} (\varrho ^{-1}(|z|+\tau )) ,\quad a\in \mathbb{R},\; |t|\le \tau .
\end{align*}

Identity \eqref{ti} shows that $\tilde{\Lambda}_a^t$ is Fr\'echet differentiable in a neighborhood of the origin with
\[
d\tilde{\Lambda}_a^t(0)=d\tilde{\Lambda}_{a^t}(0).
\]
Furthermore, \eqref{c9.0} and Proposition \eqref{propositionc1} with $a=a^t$ gives
\begin{equation}\label{ude}
\sup_{|t|\le \tau}\|d\tilde{\Lambda}_a^t(0)\|_{\mathrm{op}}\le C_\tau ^0.
\end{equation}

Now, we easily get by applying \eqref{st}, with $a_1$ and $a_2$ substituted by $a_1^t$ and $a_2^t$,
\[
|a_1(t)-a_2(t)|\le C_\tau ^1 \|d\tilde{\Lambda}_{a_1}^t(0)-d\tilde{\Lambda}_{a_2}^t(0)\|_{\mathrm{op}},\quad |t|\le \tau.
\]
That is we have
\[
\|a_1-a_2\|_{C([-\eta ,\eta])}\le C_\tau^1 \sup_{|t|\le \eta}\|d\tilde{\Lambda}_{a_1}^t(0)-d\tilde{\Lambda}_{a_2}^t(0)\|_{\mathrm{op}},
\]
which is the expected inequality.

\appendix
\section{Technical elementary lemmas}\label{appendix}

Let $\mathfrak{A}=(\mathfrak{a}^{ij})\in L^\infty (\mathbb{R}^n,\mathbb{R}^{n\times n})$ satisfying 
\[
2\beta |\xi|^2\le \mathfrak{A}(x)\xi \cdot \xi ,\quad x,\xi \in \mathbb{R}^n ,
\]
for some $\beta >0$, and $P\in L^\infty (\mathbb{R}^n,\mathbb{R}^n)$. 

Pick a bounded domain $\Omega$ of $\mathbb{R}^n$ and  consider on $H_0^1(\Omega)\times H_0^1(\Omega)$ the continuous bilinear forms
\begin{align*}
&\mathfrak{b}(u,v)=\int_\Omega [\mathfrak{A}\nabla u+uP]\cdot \nabla v,\quad u,v\in H_0^1(\Omega).
\\
&\mathfrak{b}^\ast (u,v)=\int_\Omega [\mathfrak{A}\nabla u\cdot \nabla v-vP\cdot \nabla u],\quad u,v\in H_0^1(\Omega).
\end{align*}

Denote by $\mu_\Omega$ the Poincar\'e's constant of $\Omega$:
\[
\|u\|_{L^2(\Omega)}\le \mu_\Omega \|\nabla u\|_{L^2(\Omega)},\quad u\in H_0^1(\Omega).
\]

\begin{lemma}\label{lemmaap1}
Under the assumption $\|P\|_{L^\infty (\mathbb{R}^n)}\le \mu_\Omega^{-1} \beta$, we have
\begin{align}
&\mathfrak{b}(u,u)\ge \beta \|\nabla u\|_{L^2(\Omega)},\quad u\in H_0^1(\Omega), \label{ap1}
\\
&\mathfrak{b}^\ast(u,u)\ge \beta \|\nabla u\|_{L^2(\Omega)},\quad u\in H_0^1(\Omega). \label{ap2}
\end{align}
\end{lemma}

\begin{proof}
Let $u\in H_0^1(\Omega)$. As 
\[
\int_\Omega uP\cdot \nabla u\le \|P\|_{L^\infty (\mathbb{R}^n)}\|u\|_{L^2(\Omega)}\|\nabla u\|_{L^2(\Omega)}\le \mu_\Omega \|P\|_{L^\infty (\mathbb{R}^n)}\|\nabla u\|_{L^2(\Omega)}^2,
\]
we get
\[
\mathfrak{b}(u,u)\ge \int_\Omega \left(2\beta - \mu_\Omega \|P\|_{L^\infty (\mathbb{R}^n)}\right)\|\nabla u\|_{L^2(\Omega)}.
\]
Therefore \eqref{ap1} follows. The proof of \eqref{ap2} is similar.
\end{proof}

We introduce a definition. Let $f\in H^{1/2}(\Gamma)$. We say that $u\in H^1(\Omega)$ is a weak solution of the BVP
\begin{equation}\label{ap3}
\left\{
\begin{array}{ll}
\mathrm{div}(\mathfrak{A}\nabla u+uP)=0\quad \mathrm{in}\; \Omega,
\\
u_{|\Gamma}=f
\end{array}
\right.
\end{equation}
if $u_{|\Gamma}=f$ (in the trace sense) and 
\[
\mathfrak{b}(u,v)=0,\quad v\in H_0^1(\Omega).
\]
Note that this last condition implies that the first equation in \eqref{ap3} holds in $H^{-1}(\Omega)$.

\vskip.1cm
Also, we say that $u^\ast\in H^1(\Omega)$ is a weak solution of the BVP
\begin{equation}\label{ap4}
\left\{
\begin{array}{ll}
\mathrm{div}(\mathfrak{A}\nabla u^\ast)-P\cdot \nabla u^\ast=0\quad \mathrm{in}\; \Omega,
\\
u^\ast_{|\Gamma}=f,
\end{array}
\right.
\end{equation}
if $u^\ast_{|\Gamma}=f$ and 
\[
\mathfrak{b}^\ast(u^\ast,v)=0,\quad v\in H_0^1(\Omega).
\]

Let us assume that $\mathfrak{A}$ satisfies in addition
\[
\max_{1\le i,j\le n}\|\mathfrak{a}^{ij}\|_{L^\infty (\mathbb{R}^n)}\le \tilde{\beta},
\]
for some $\tilde{\beta}>0$.

Let $\mathcal{E}f$ be the unique element of $\dot{f}$ so that $\|\mathcal{E}f\|_{H^1(\Omega)}=\|f\|_{H^{1/2}(\Gamma)}$ ($\mathcal{E}f$ is nothing but the orthogonal projection of $0\in H^1(\Omega)$ on the closed convex set $\dot{f}$). Then 
\begin{equation}\label{ap5}
\|\mathrm{div}(\mathfrak{A}\nabla \mathcal{E}f+\mathcal{E}fP)\|_{H^{-1}(\Omega)}\le \left(\tilde{\beta}+\|P\|_{L^\infty(\mathbb{R}^n)}\right)\|f\|_{H^{1/2}(\Gamma)}.
\end{equation}
Furthermore, we have
\begin{equation}\label{ap6}
\langle \mathrm{div}(\mathfrak{A}\nabla \mathcal{E}f+\mathcal{E}fP),v\rangle_{-1}=\mathfrak{b}(\mathcal{E}f,v),\quad v\in H_0^1(\Omega),
\end{equation}
where $\langle \cdot,\cdot\rangle_{-1}$ is the duality pairing between $H_0^1(\Omega)$ and $H^{-1}(\Omega)$.

\begin{lemma}\label{lemmaap2}
Suppose that $\|P\|_{L^\infty (\mathbb{R}^n)}\le \mu_\Omega^{-1} \beta$. Then we have
\\
(i) the BVP \eqref{ap3} admits a unique weak solution $u\in H^1(\Omega)$ satisfying
\begin{equation}\label{ap7}
\|u\|_{H^1(\Omega)}\le C\|f\|_{H^{1/2}(\Gamma)},
\end{equation}
and 
\\
(ii) the BVP \eqref{ap4} has a unique weak solution $u^\ast \in H^1(\Omega)$ satisfying
\begin{equation}\label{ap8}
\|u^\ast\|_{H^1(\Omega)}\le C\|f\|_{H^{1/2}(\Gamma)},
\end{equation}
where $C=C(n,\Omega, \beta,\tilde{\beta})>0$.
\end{lemma}

\begin{proof}
We provide the proof of (i) and omit that of (ii) which is quite similar to that of (i).

\vskip.1cm
Since $\mathfrak{b}$ is coercive by Lemma \ref{lemmaap1}, in light of \eqref{ap6} we get by applying Lax-Milgram's lemma that there exists a unique $u_0\in H_0^1(\Omega)$ satisfying
\[
\mathfrak{b}(u_0,v)=-\mathfrak{b}(\mathcal{E}f,v),\quad v\in H_0^1(\Omega).
\]
In particular we have
\[
\mathfrak{b}(u_0,u_0)=-\mathfrak{b}(\mathcal{E}f,u_0),\quad v\in H_0^1(\Omega).
\]
Using \eqref{ap1}, \eqref{ap5} and \eqref{ap6}, we derive
\begin{equation}\label{ap9}
\beta \|\nabla u_0\|_{L^2(\Omega)}\le \left(\tilde{\beta}+\|P\|_{L^\infty(\mathbb{R}^n)}\right)\|f\|_{H^{1/2}(\Gamma)}.
\end{equation}
Clearly, $u=u_0+\mathcal{E}f$ satisfies 
\[
\mathfrak{b}(u,v)=0,\quad v\in H_0^1(\Omega).
\]
and $u_{|\Gamma}=f$. In other words $u$ is a weak solution of \eqref{ap3} and, as a consequence of \eqref{ap9}, $u$ satisfies \eqref{ap7}. 

\vskip.1cm
We complete the proof by noting that the uniqueness of solutions of \eqref{ap3} is a straightforward consequence of the coercivity of $\mathfrak{b}$.
\end{proof}

Consider now the BVP
\begin{equation}\label{ap10}
\left\{
\begin{array}{ll}
\mathrm{div}(\mathfrak{A}\nabla u+uP)=\mathrm{div}(F)\quad \mathrm{in}\; \Omega,
\\
u_{|\Gamma}=f,
\end{array}
\right.
\end{equation}
where $f\in H^{1/2}(\Gamma)$ and $F\in L^2(\Omega)^n$.

\vskip.1cm
Assume first that $f=0$. In that case the variational problem associated to \eqref{ap10} has the form
\begin{equation}\label{ap11}
\mathfrak{b}(u,v)=\langle \mathrm{div}F,v\rangle_{-1} ,\quad v\in H_0^1(\Omega),
\end{equation}

As in the preceding proof we show, with the help of Lax-Milgram's lemma, that the variational problem \eqref{ap11} admits a unique solution $u(F)\in H_0^1(\Omega)$ satisfying
\[
\|u(F)\|_{H^1(\Omega)}\le C\|F\|_{L^2(\Omega)^n},
\] 
where $C=C(n,\Omega, \beta,\tilde{\beta})>0$.

\vskip.1cm
By linearity, the solution of \eqref{ap10} is the sum of the solution of \eqref{ap10} with $f=0$ and the solution of \eqref{ap10} with $F=0$ (corresponding to \eqref{ap3}).

\vskip.1cm
We derive from  Lemma \ref{lemmaap2} the following result.

\begin{lemma}\label{lemmaap3}
Suppose that $\|P\|_{L^\infty (\mathbb{R}^n)}\le \mu_\Omega^{-1} \beta$. Let $f\in H^{1/2}(\Gamma)$ and $F\in L^2(\Omega)^n$. Then the BVP \eqref{ap10} admits a unique weak solution $u\in H^1(\Omega)$ satisfying
\begin{equation}\label{ap13}
\|u\|_{H^1(\Omega)}\le C\left(\|f\|_{H^{1/2}(\Gamma)}+\|F\|_{L^2(\Omega)^n}\right),
\end{equation}
where $C=C(n,\Omega, \beta,\tilde{\beta})>0$.
\end{lemma}

Next, consider the BVP
\begin{equation}\label{ap12}
\left\{
\begin{array}{ll}
\mathrm{div}(\mathfrak{A}\nabla u^\ast)-P\cdot \nabla u^\ast=R\cdot \nabla g\quad \mathrm{in}\; \Omega,
\\
u^\ast_{|\Gamma}=f,
\end{array}
\right.
\end{equation}

Assume that $R\in W^{1,\infty}(\Omega)$ with $\|R\|_{W^{1,\infty}(\Omega)}\le \varrho$, for some $\varrho>0$, and $g\in H^1(\Omega)$. A simple integration by parts enables us to show that
\[
\|R\cdot \nabla g\|_{H^{-1}(\Omega)}\le c\|g\|_{L^2(\Omega)},
\]
where $c=c(n,\Omega,\varrho)$.

\vskip.1cm
With the help of this inequality we can proceed as above in order to derive the following lemma.

\begin{lemma}\label{lemmaap4}
Suppose that $\|P\|_{L^\infty (\mathbb{R}^n)}\le \mu_\Omega^{-1} \beta$. Let $g\in H^1(\Omega)$ and  $R\in W^{1,\infty}(\Omega)$ satisfying $\|R\|_{W^{1,\infty}(\Omega)}\le \varrho$, for some $\varrho>0$. Then the BVP \eqref{ap12} has a unique weak solution $u^\ast \in H^1(\Omega)$ satisfying
\begin{equation}\label{ap14}
\|u^\ast\|_{H^1(\Omega)}\le C\left(\|f\|_{H^{1/2}(\Gamma)}+\|g\|_{L^2(\Omega)}\right),
\end{equation}
where $C=C(n,\Omega, \beta,\tilde{\beta},\varrho)>0$.
\end{lemma}

\vskip.2cm

\end{document}